\documentclass[10pt]{amsart}
\usepackage{verbatim}
\usepackage{eucal,url,amssymb,stmaryrd,enumerate,amscd,verbatim}
\usepackage{amsmath}
\usepackage[pagebackref,colorlinks=true,linkcolor=blue,citecolor=blue]{hyperref}
\usepackage{amsfonts}
\usepackage{amsthm}
\usepackage[margin=1in]{geometry}

\allowdisplaybreaks
\numberwithin{equation}{section}
\newtheorem{thrm}{Theorem}[section]
\newtheorem{lemma}[thrm]{Lemma}
\newtheorem{prop}[thrm]{Proposition}
\newtheorem{cor}[thrm]{Corollary}
\newtheorem{dfn}[thrm]{Definition}

\newtheorem{conv}[thrm]{Convention}

\newcommand{\vnab}{\nabla_v}
\newcommand{\bnab}{\nabla_b}
\newcommand{\eqdef}{\overset{def}{=}}
\newcommand{\partt}{\frac{\partial}{\partial t}}
 1

\def\bi{\nabla}

\newcommand{\vol}{\, Vol_{\eta}}

\newcommand{\Ncal}{\mathcal{N}}
\newcommand{\Wcal}{\mathcal{W}}

\begin{document}

\begin{abstract}
We establish in the present paper two sub-gradient estimates for the quaternionic contact (qc) heat equation on a compact qc manifold of dimension $4n+3$, provided some positivity conditions are satisfied. These are qc versions of the prominent Li-Yau gradient estimate in Riemannian geometry. 
Another goal of this paper is to get two Perelman-type entropy formulas for the qc heat equation on a compact qc-Einstein manifold of dimension $4n+3$ with non-negative qc scalar curvature (e.g. compact $3$-Sasakian manifold), as well as an integral sub-gradient estimate for the positive solutions of the qc heat equation.
\end{abstract}

\title[Li-Yau sub-gradient estimates and Perelman-type entropy formulas...]{Li-Yau sub-gradient estimates and Perelman-type entropy formulas for the heat equation in quaternionic contact geometry}
\date{\today }
\author{Stefan Ivanov}
\address[Stefan Ivanov]{University of Sofia, Faculty of Mathematics and
Informatics, blvd. James Bourchier 5, 1164, Sofia, Bulgaria}
\address{and Institute of Mathematics and Informatics, Bulgarian Academy of
Sciences} \email{ivanovsp@fmi.uni-sofia.bg}
\author{Alexander Petkov}
\address[Alexander Petkov]{University of Sofia, Faculty of Mathematics and
Informatics, blvd. James Bourchier 5, 1164, Sofia, Bulgaria}
\email{a\_petkov\_fmi@abv.bg}

\maketitle

\tableofcontents

\setcounter{tocdepth}{2}

\section{Introduction}
In the seminal paper  \cite{LY}, P. Li and S.-T. Yau established the
parabolic Li-Yau gradient estimate and Harnack inequality for the positive solution of the heat equation in a complete Riemannian manifold with nonnegative Ricci curvature. 

In his fundamental  paper G. Perelman \cite{Per} derived an entropy formula for Ricci flow. The
formula turns out being of fundamental importance in the study of Ricci flow (cf. \cite[Sections 3, 4, 10]{Per}). The derivation of the entropy
formula in \cite[Section 9]{Per} resembles the gradient estimate for the linear heat equation
proved by Li-Yau  on the linear parabolic equation.

The corresponding problems in CR-geometry, namely the sub-parabolic Li-Yau CR gradient estimate and the Perelman CR entropy formula  are developed in \cite{CKL}.

Here we investigate the quaternionic contact (qc) counterpart of the Li-Yau gradient estimate and the Perelman's entropy.

Quaternionic contact geometry is an example of sub-Riemannian geometry.  A quaternionic
contact (qc) structure, \cite{Biq1}, appears naturally as the
conformal boundary at infinity of the quaternionic hyperbolic
space and it is further developed in connection with finding  the extremals and the best constant in the $L^2$ Folland-Stein inequality on the quaternionic Heisenberg group and related qc Yamabe problem \cite{F2,FS,IMV,IMV2,IP,IMV3,IV}.

\emph{A quaternionic contact structure} 
(qc structure)  on a real (4n+3)-dimensional manifold $M$ is a codimension three distribution $H$ (\emph{the horizontal distribution}) locally given as the kernel of an $
\mathbb{R}^3$-valued one-form $\eta=(\eta_1,\eta_2,\eta_3)$, such
that the three two-forms $d\eta_i|_H$ are the fundamental forms of
a quaternionic Hermitian structure on $H$.
Any quaternionic contact manifold  admits a canonical connection $\nabla$ preserving the qc structure and having minimal torsion, called  the Biquard connection, which role in qc geometry is similar to  the Levi--Civita connection in Riemannian geometry and Tanaka--Webster connection in CR geometry. 

The basic concrete examples of qc-manifolds are provided by the extensively studied
3-Sasakian spaces and the quaternionic  Heisenberg group. As well known
\cite{BGM}, see also \cite{BG} for a recent complete account, 3-Sasakian manifolds are characterized as
Riemannian manifolds whose cone is a hyper-Kaehler manifold. These spaces are Einstein with positive Riemannian scalar curvature \cite{Kas}.

From the point of view of qc geometry, 3-Sasakian structures are qc manifolds whose torsion endomorphism of the Biquard connection vanishes. In turn, the latter property is equivalent to the qc structure being qc-Einstein,
i.e., the trace-free part of the qc-Ricci tensor vanishes, see \cite{IMV}. As a consequence of the second Bianchi identity for the Biquard connection it follows that for any qc-Einstein manifold the horizontal scalar curvature $S$ of the Biquard connection (the qc scalar curvature) is constant and if $S>0$ the qc-Einstein space is locally 3-Sasakian \cite{IMV,IMV1}.

We also recall that complete and regular 3-Sasakian spaces have canonical fibering with fiber $Sp(1)$ or $SO(3)$ and base a quaternionic Kaehler manifold of positive scalar curvature \cite{BGM,BG}.  It is shown in \cite{IV3} that if the qc scalar curvature of a regular qc-Einstein manifold is negative, $S<0$, 
then the space is ''essentially'' an $SO(3)$ bundle over quaternionic Kaehler manifold with negative scalar curvature,  described in \cite{Tan,Jel}. Similarly,  in the ''regular'' case,  a qc-Einstein manifold of zero scalar curvature, $S=0$, fibers over a
hyper-Kaehler manifold, see \cite[Proposition~6.3]{IMV1}. In particular, the quaternionic Heisenberg group is an $SO(3)$ bundle over the flat space.

Let $(M,g,\mathbb{Q})$ be a compact qc manifold. We consider the qc heat equation  \cite{Vas,Pet1,Pet2}
\begin{equation}\label{qcheat}
\frac{\partial}{\partial t}u=-\Delta_b u,
\end{equation}
where $u(x,t):M\times[0,\infty)\rightarrow\mathbb{R}$ is a smooth function and $\Delta_b:\mathcal{F}(M)\rightarrow\mathcal{F}(M)$ is the sub-Laplacian on $M.$

Let $$\vnab\varphi=\sum_{s=1}^3d\varphi(\xi_s)\xi_s$$ be   the vertical gradient of a function $\varphi\in\mathcal{F}(M)$. Our main result is the next 
\begin{thrm}\label{thrm1}
Let $(M,g,\mathbb{Q})$ be a compact qc manifold of dimension $4n+3$ and the positivity condition 
\begin{equation}\label{Pos}
2(n+2)Sg(X,X)+2nT^0(X,X)+4(n+4)U(X,X)\geq 0
\end{equation}
holds. Suppose that $u(x,t)$ is a positive solution of \eqref{qcheat}, satisfying
\begin{equation}\label{commut}
(\nabla^3u)(e_a,e_a,\nabla_v u)=-\vnab u(\Delta_b u).
\end{equation}
For $f(x,t)=\ln u(x,t)$  the following sub-gradient estimate holds
\begin{equation}\label{sub1}
|\bnab f|^2-\alpha\partt f+\frac{8n}{3}t|\vnab f|^2\leq\frac{2n\alpha^2}{t}, \quad \alpha=\frac{9+2n}{2n},
\end{equation}
where we apply Convention~\ref{conv}  about summation rules, and the torsion tensors $T^0$ and $U$, as well as the normalized qc scalar curvature $S,$ are defined below in \eqref{Tcompnts} and \eqref{qscs}. 
\end{thrm}

As a consequence of Theorem~\ref{thrm1} we get the next
\begin{cor}\label{cor1}
Let $(M,g,\mathbb{Q})$ be a compact  qc-Einstein manifold of dimension $4n+3$ with non-negative constant qc scalar curvature, $S\ge 0$. Supose that $u(x,t)$ is a positive solution of \eqref{qcheat}. Then $f(x,t)=\ln u(x,t)$ satisfies the sub-gradient estimate \eqref{sub1}.

In particular, on a compact  3-Sasakian manifold the function  $f(x,t)$  satisfies the sub-gradient estimate \eqref{sub1}.
\end{cor}

Our second main result concerns the case when the a priori lower bound \eqref{Pos} of Theorem~\ref{thrm1} is replaced by a negative constant, namely, we state the 
\begin{thrm}\label{thrm2}
Let $(M,g,\mathbb{Q})$ be a compact qc manifold of dimension $4n+3$ satisfying the positivity condition
\begin{equation}\label{PosA}
2(n+2)Sg(X,X)+2nT^0(X,X)+4(n+4)U(X,X)\geq -kg(X,X),
\end{equation}
where $k>0$ is a constant. Suppose that $u(x,t)$ is a positive solution of \eqref{qcheat}, satisfying \eqref{commut}. Then the following sub-gradient estimate for  $f(x,t)=\ln u(x,t)$ holds:
\begin{equation}\label{sub2}
|\bnab f|^2-\alpha_k\partt f<\frac{n\alpha_k^2(k+2)}{t}+\frac{8n^2\alpha_k^2(k+2)}{k+3}, \quad \alpha_k=\frac{9+2n(k+1)}{2n}.
\end{equation}
\end{thrm}

As a simple consequence of Theorem~\ref{thrm2} we obtain the following

\begin{cor}\label{cor3}
Let $(M,g,\mathbb{Q})$ be a compact  qc-Einstein manifold of dimension $4n+3$ with negative qc scalar curvature $S=-\frac{k}{2(n+2)},$ where $k>0$ is a constant. Suppose that $u(x,t)$ is a positive solution of \eqref{qcheat}. Then $f(x,t)=\ln u(x,t)$ satisfies the sub-gradient estimate \eqref{sub2}.
\end{cor}

Furthermore, following the Riemannian and CR cases \cite{Per,CKL}, we define the \emph{Nash--type functionals}
\begin{gather}\label{Nash1}
\Ncal(u,t)=-\int_M(\ln u)u\vol,\\
\tilde{\Ncal}(u,t)=\Ncal(u,t)-2n\alpha^2a[\ln(4\pi t)+1]\label{Nash2},
\end{gather}
as well as the \emph{Perelman--type functionals}
\begin{gather}\label{Perelman1}
\Wcal(u,t)=\int_M[t|\bnab\varphi|^2+\varphi-4n\alpha^2a]u\vol,
\\\label{Perelman2}
\tilde{\Wcal}(u,t)=\Wcal(u,t)+4nt^2\int_M|\vnab\varphi|^2u\vol.
\end{gather}
In the expressions \eqref{Nash1}, \eqref{Nash2}, \eqref{Perelman1} and \eqref{Perelman2} $u(x,t)$ is a positive solution of \eqref{qcheat}, normalized as 
\begin{equation}\label{normu}
\int_Mu\vol=1, \quad  \quad \textnormal{expressed by} 
\end{equation}
\begin{equation}\label{heatkern}
u(x,t)=\frac{e^{-\varphi(x,t)}}{(4\pi t)^{2n\alpha^2a}}
\end{equation} 
for a suitable function $\varphi(x,t):M\times[0,\infty)\rightarrow\mathbb{R},$ where $\alpha=\frac{9+2n}{2n}$ and $a$ is a real constant to be determined. 
As a consequence of  Corollary~\ref{cor1} we obtain two Perelman-type entropy formulas concerning $\tilde{\Ncal}(u,t)$ and $\tilde{\Wcal}(u,t).$ 

Namely, the first one states as follows
\begin{thrm}\label{thrm3}
Let $(M,g,\mathbb{Q})$ be a compact  qc-Einstein $(4n+3)$-dimensional manifold  with  non-negative qc scalar curvature. Let  $u(x,t)$ be a positive solution of \eqref{qcheat}, satisfying \eqref{normu}. Then we have
\begin{equation}\label{Nashentropy}
\frac{d}{dt}\tilde{\Ncal}(u,t)=\int_M\Big[|\bnab\varphi|^2+\alpha\partt\varphi+\frac{2n(\alpha-1)\alpha^2a}{t}\Big]u\vol\leq 0
\end{equation}
for $t\in(0,\infty)$ and $a\geq1.$

In particular, on a compact  3-Sasakian manifold any positive solution to the qc heat equation \eqref{qcheat} normalized by \eqref{normu} satisfies the inequality \eqref{Nashentropy}.
\end{thrm}
As a consequence of Theorem~\ref{thrm3} we obtain the following integral version of the sub-gradient estimate:
\begin{cor}\label{cor2}
Let $(M,g,\mathbb{Q})$ be a compact  qc-Einstein manifold of dimension $4n+3$ with  non-negative qc scalar curvature. Suppose that $u(x,t)$ is a positive solution of \eqref{qcheat}. Then there exists a positive constant $C$ s.t. the following estimate holds:
\begin{equation}\label{integralsub}
\int_M|\bnab u^\frac12|^2\vol\leq\frac{C}{t}.
\end{equation}
In particular,  on a compact  3-Sasakian manifold any positive solution $u(x,t)$ of the qc heat equation \eqref{qcheat} satisfies the integral estimate \eqref{integralsub}.
\end{cor}

The second one is the content of the following
\begin{thrm}\label{thrm4}
Let $(M,g,\mathbb{Q})$ be a compact  qc-Einstein manifold of dimension $4n+3$ with  non-negative qc scalar curvature. Let $u(x,t)$ be a positive solution of \eqref{qcheat}, satisfying \eqref{normu}. 
Then we have
\begin{multline}\label{Perelmanentropy}
\frac{d}{dt}\tilde{\Wcal}(u,t)\leq-2t\int_Mu|(\nabla^2\varphi)_{[-1][sym]}|^2\vol\\-\frac{t}{2n}\int_Mu(\Delta_b\varphi)^2\vol-4(n+2)St\int_Mu|\bnab\varphi|^2\vol+\frac{2(2n-na+6)\alpha^2}{t}\leq 0
\end{multline}
for $t\in(0,\infty)$ and $a\geq\frac{2n+6}{n}.$

In particular, on a compact  3-Sasakian manifold any positive solution to the qc heat equation \eqref{qcheat} normalized by \eqref{normu} satisfies the inequality \eqref{Perelmanentropy}.
\end{thrm}

\begin{conv}\label{conv}
\label{conven} \hfill\break\vspace{-15pt}

\begin{enumerate}[ a)]

\item We shall use $X,Y,Z,U$ to denote horizontal vector fields, i.e. $%
X,Y,Z,U\in H$.

\item $\{e_1,\dots,e_{n},I_1e_1,\dots,I_1e_{n},I_2e_1,\dots,I_2e_{n},I_3e_1,\dots,I_3e_{n}\}$ denotes an adapted  local 
orthonormal basis of the horizontal space $H$.

\item The summation convention over repeated vectors/indices  from the basis
$\{e_1,\dots,e_{4n}\}$ will be used. For example, for a
(0,2)-tensor $P$ we have
\begin{multline*}
P(e_b,e_b)=\sum_{b=1}^{4n}P(e_b,e_b)=\sum_{b=1}^{n}P(e_b,e_b)+\sum_{b=1}^{n}P(I_1e_b,I_1e_b)
+\sum_{b=1}^{n}P(I_2e_b,I_2e_b)+\sum_{b=1}^{n}P(I_3e_b,I_3e_b).
\end{multline*}

\item The triple $(i,j,k)$ denotes any cyclic permutation of
$(1,2,3)$.

\item $s$ and $t$ will be any numbers from the set $\{1,2,3\}$,
$s, t\in\{1,2,3\} $.

\end{enumerate}
\end{conv}

\textbf{Acknowledgments}
The research of S.I. is partially supported by the European Union-Next Generation EU, through the National Recovery and Resilience Plan of the Republic of Bulgaria, project 
SUMMIT BG-RRP-2.004-0008-C01. The research of A.P. is partially financed  by Contract KP-06-H72-1/05.12.2023 with the National Science Fund of Bulgaria, Contract 80-10-181/22.4.2024    with the Sofia University "St. Kl. Ohridski" and  the National Science Fund of Bulgaria, National Scientific Program ``VIHREN", Project No. KP-06-DV-7. A.P. also acknowledges the University of Miami (UM), during his stay in which as a Fulbright Scholar under the auspices of the UM's agreement with Fulbright Bulgaria, the first version of the manuscript was prepared.

\section{Quaternionic contact manifolds}

Quaternionic contact manifolds,  introduced by O. Biquard in 
\cite{Biq1}, appear naturally as the
conformal boundary at infinity of the quaternionic hyperbolic
space and it is further developed in connection with finding  the extremals and the best constant in the $L^2$ Folland-Stein inequality on the quaternionic Heisenberg group and related qc Yamabe problem \cite{IMV,IMV2,IP,IMV3,IV}.

\subsection{Quaternionic contact structures and the Biquard connection} 
\label{ss:Biq conn}

Following Biquard, a quaternionic contact (qc) manifold $(M, g, \mathbb{Q})$ is a $4n+3$%
-dimensional manifold $M$ with a codimension three distribution $H$ equipped
with an $Sp(n)Sp(1)$--structure. Explicitly, $H$ is the kernel of a local
1-form $\eta=(\eta_1,\eta_2,\eta_3)$ with values in $\mathbb{R}^3$ together
with a compatible Riemannian metric $g$ and a rank-three bundle $\mathbb{Q}$
consisting of endomorphisms of $H$ locally generated by three almost complex
structures $I_1,I_2,I_3$ on $H$ satisfying the identities of the imaginary
unit quaternions. Thus, we have $I_1I_2=-I_2I_1=I_3, \quad
I_1I_2I_3=-id_{|_H},$  hermitian compatible with the metric, $%
g(I_s.,I_s.)=g(.,.),$ and the next  conditions
hold $2g(I_sX,Y)\ =\ d\eta_s(X,Y)$.

On a qc manifold of dimension $(4n+3)>7$ with a fixed metric $g$
on  $H$ there exists a canonical connection defined in
\cite{Biq1}. Biquard  showed that there is a unique connection $%
\nabla$ with torsion $T$ and a unique supplementary subspace $V$ to $H$ in $%
TM$, such that:
\begin{enumerate}[(i)]
\item $\nabla$ preserves the splitting $H\oplus V$ and the
$Sp(n)Sp(1)$--structure on $H$, 
\[\nabla g=0,\qquad \nabla I_i=-\alpha_jI_k+\alpha_kI_j
\]
and its torsion on $H$ is given by
$T(X,Y)=-[X,Y]_{|V}$;
\item for $\xi\in V$, the endomorphism $T(\xi,.)_{|H}$ of $H$ lies in $%
(sp(n)\oplus sp(1))^{\bot}\subset gl(4n)$; \item the connection on
$V$ is induced by the natural identification 
of $V$ with $\mathbb Q$, 
$\nabla\varphi=0$,
\[\nabla\xi_i=-\alpha_j\xi_k+\alpha_k\xi_j.\]
\end{enumerate}


When the dimension of $M$ is at least eleven \cite{Biq1} also
described the supplementary \emph{vertical distribution} $V$,
which is (locally) generated by the so called \emph{Reeb vector
fields} $\{\xi_1,\xi_2,\xi_3\}$ determined by
\begin{equation}  \label{bi1}
\begin{aligned}
\eta_s(\xi_k)=\delta_{sk}, \qquad (\xi_s\lrcorner d\eta_s)_{|H}=0,
\qquad (\xi_s\lrcorner d\eta_k)_{|H}=-(\xi_k\lrcorner
d\eta_s)_{|H},
\end{aligned}
\end{equation}
where $\lrcorner$ denotes the interior multiplication.

If the dimension of $M $ is seven, Duchemin shows in \cite{D} that if we
assume, in addition, the existence of Reeb vector fields as in \eqref{bi1},
then the Biquard result holds. \emph{Henceforth, by a qc structure in
dimension $7$ we shall mean a qc structure satisfying \eqref{bi1}}. This
implies the existence of the connection with properties (i), (ii) and (iii)
above.

The fundamental 2-forms $\omega_s$ of the quaternionic contact structure  are
defined by
\begin{equation*}
2\omega_{s|H}\ =\ \, d\eta_{s|H},\qquad \xi\lrcorner\omega_s=0,\quad \xi\in
V.
\end{equation*}
The torsion of the Biquard connection restricted to $H$ has the form 
\begin{equation*}
T(X,Y)=-[X,Y]_{|V}=2\sum_{s=1}^3\omega_s(X,Y)\xi_s.
\end{equation*}

\subsection{Invariant decompositions}

Any endomorphism $\Psi$ of $H$ can be decomposed with respect to the
quaternionic structure $(\mathbb{Q},g)$ uniquely into four $Sp(n)$--invariant
parts, 
$\Psi=\Psi^{+++}+\Psi^{+--}+\Psi^{-+-}+\Psi^{--+},$ 
where $\Psi^{+++}$ commutes with all three $I_i$, $\Psi^{+--}$ commutes with
$I_1$ and anti-commutes with the others two, etc. The two $Sp(n)Sp(1)$%
--invariant components \index{$Sp(n)Sp(1)$-invariant
components!$\Psi_{[3]}$} \index{$Sp(n)Sp(1)$-invariant
components!$\Psi_{[-1]}$} are given by $
\Psi_{[3]}=\Psi^{+++},\quad
\Psi_{[-1]}=\Psi^{+--}+\Psi^{-+-}+\Psi^{--+}. $ These are the projections on the eigenspaces of the Casimir
operator $
\Upsilon =\ I_1\otimes I_1\ +\ I_2\otimes I_2\ +\ I_3\otimes I_3,
$ corresponding, respectively, to the eigenvalues $3$ and $-1$, see \cite{CSal}%
. Note here that each of the three 2-forms $\omega_s$ belongs to the
[-1]-component, $\omega_s=\omega_{s[-1]}$, and constitute a basis of the Lie
algebra $sp(1)$.

If $n=1$ then the space of symmetric endomorphisms commuting with all $I_s$
is 1-dimensional, i.e., the [3]-component of any symmetric endomorphism $\Psi
$ on $H$ is proportional to the identity, $\Psi_{[3]}=-%
\frac{tr\Psi}{4}Id_{|H}$.

\subsection{The torsion tensor}

The torsion endomorphism $T_{\xi }=T(\xi ,\cdot ):H\rightarrow H,\quad \xi
\in V,$ will be decomposed into its symmetric part $T_{\xi }^{0}$ and
skew-symmetric part $b_{\xi }, T_{\xi }=T_{\xi }^{0}+b_{\xi }$. Biquard
showed in \cite{Biq1} that the torsion $T_{\xi }$ is completely trace-free, $%
tr\,T_{\xi }=tr\,T_{\xi }\circ I_{s}=0$, its symmetric part has the
properties $T_{\xi _{i}}^{0}I_{i}=-I_{i}T_{\xi _{i}}^{0},\quad I_{2}(T_{\xi
_{2}}^{0})^{+--}=I_{1}(T_{\xi _{1}}^{0})^{-+-},\quad I_{3}(T_{\xi
_{3}}^{0})^{-+-}=I_{2}(T_{\xi _{2}}^{0})^{--+},\quad I_{1}(T_{\xi
_{1}}^{0})^{--+}=I_{3}(T_{\xi _{3}}^{0})^{+--}$.
The skew-symmetric part can be represented as $b_{\xi _{i}}=I_{i}u$, where $u
$ is a traceless symmetric (1,1)-tensor on $H$ which commutes with $%
I_{1},I_{2},I_{3}$. Therefore we have $T_{\xi _{i}}=T_{\xi _{i}}^{0}+I_{i}u$%
. When $n=1$ the tensor $u$ vanishes identically, $u=0$, and the
torsion is a symmetric tensor, $T_{\xi }=T_{\xi }^{0}.$

The two $Sp(n)Sp(1)$--invariant trace-free symmetric 2-tensors on $H$
\begin{equation} \label{Tcompnts}
T^0(X,Y)= g((T_{\xi_1}^{0}I_1+T_{\xi_2}^{0}I_2+T_{ \xi_3}^{0}I_3)X,Y) \
\text{ and }\ U(X,Y) =g(uX,Y)
\end{equation}
were introduced in \cite{IMV} and enjoy the properties
\begin{equation}  \label{propt}
\begin{aligned} T^0(X,Y)+T^0(I_1X,I_1Y)+T^0(I_2X,I_2Y)+T^0(I_3X,I_3Y)=0, \\
U(X,Y)=U(I_1X,I_1Y)=U(I_2X,I_2Y)=U(I_3X,I_3Y),\\
T^0(e_a,e_a)=T^0(I_se_a,e_a)=U(e_a,e_a)=0. \end{aligned}
\end{equation}
\indent It is shown in \cite[Proposition~2.3]{IV} 
$
4T^0(\xi_s,I_sX,Y)=T^0(X,Y)-T^0(I_sX,I_sY)
$ 
which yields (see e.g. \cite[formula (2.5)]{IPV3})
\begin{equation}\label{need1}
T(\xi_s,I_sX,Y)
=\frac14\Big[T^0(X,Y)-T^0(I_sX,I_sY)\Big]-U(X,Y).
\end{equation}
In dimension seven $(n=1)$ the tensor $U$ vanishes identically,
$U=0$.
\subsection{Torsion and curvature}
Let $R=[\nabla,\nabla]-\nabla_{[\ ,\ ]}$ be the curvature tensor of $\nabla$
and the dimension is $4n+3$. We denote the curvature tensor of type (0,4)
and the torsion tensor of type (0,3) by the same letter, $%
R(A,B,C,D):=g(R(A,B)C,D),\quad T(A,B,C):=g(T(A,B),C)$, $A,B,C,D
\in \Gamma(TM)$. The  Ricci tensor, the normalized
scalar curvature and the  Ricci $2$-forms  of the Biquard connection, called \emph{qc-Ricci tensor} $Ric$,
\emph{normalized qc-scalar curvature} $S$ and \emph{qc-Ricci forms} $\rho_s, \tau_s$, respectively, are
defined by
\begin{equation}  \label{qscs}
\begin{aligned}
Ric(A,B)=R(e_b,A,B,e_b),\quad 8n(n+2)S=R(e_b,e_a,e_a,e_b),\\
\rho_s(A,B)=\frac1{4n}R(A,B,e_a,I_se_a), \quad \tau_s(A,B)=\frac1{4n}R(e_a,I_se_a,A,B).
\end{aligned}
\end{equation}
The $sp(1)$-part of $R$ is determined by the Ricci 2-forms and the
connection 1-forms by
\begin{equation}  \label{sp1curv}
R(A,B,\xi_i,\xi_j)=2\rho_k(A,B)=(d\alpha_k+\alpha_i\wedge\alpha_j)(A,B),
\qquad A,B \in \Gamma(TM).
\end{equation}
The horizontal Ricci tensor and the horizontal Ricci 2-forms can be expressed in terms of the
torsion of the Biquard connection \cite{IMV} (see also
\cite{IV,IV2}). We collect the necessary facts from
\cite[Theorem~3.12, Proposition~4.3 and Proposition~4.7]{IMV}, with slight modification presented in
\cite{IV},

\begin{thrm}\cite{IMV}\label{sixtyseven} On a $(4n+3)$-dimensional qc manifold $(M,\eta,\mathbb{Q})$ with a normalized qc scalar curvature $S$ we have the following relations
\begin{equation} \label{sixtyfour}
\begin{aligned}
& Ric(X,Y)  =(2n+2)T^0(X,Y)+(4n+10)U(X,Y)+2(n+2)Sg(X,Y),\\
& \rho_s(X,I_sY) = -\frac12\Bigl[T^0(X,Y)+T^0(I_sX,I_sY)\Bigr]-2U(X,Y)-Sg(X,Y),\\
& \tau_s(X,I_sY)  = -\frac{n+2}{2n}\Bigl[T^0(X,Y)+T^0(I_sX,I_sY)\Bigr]-Sg(X,Y),\\
& g(T(\xi_i,\xi_j),X) =-\rho_k(I_iX,\xi_i)=-\rho_k(I_jX,\xi_j)=-g([\xi_i,\xi_j],X),\\
&Ric(\xi_i,I_iX)=2\rho_k(I_jX,\xi_i)+2\rho_j(I_iX,\xi_k)+\sum_{a=1}^{4n}(\nabla_{e_a}T)(\xi_i,I_iX,e_a).
\end{aligned}
\end{equation}
For $n=1$ the above formulas hold with $U=0$.
\end{thrm}
\begin{dfn}
A qc structure is said to be qc-Einstein if the horizontal qc-Ricci tensor is a scalar multiple of the metric, $Ric(X,Y)=fg(X,Y)$.
\end{dfn} 
In view of Theorem~\ref{sixtyseven} the qc-Einstein condition takes the form $Ric(X,Y)=2(n+2)Sg(X,Y)$ and is equivalent to the vanishing of the torsion endomorphism of the
Biquard connection \cite{IMV}.  As a consequence of the second Bianchi identity for the Biquard connection one has that the qc scalar curvature $S$ of a qc-Einstein manifold is a constant and  the vertical distribution is integrable \cite{IMV,IMV1}. We have

\begin{thrm}\cite{IMV,IMV1}\label{3sas}
Any 3-Sasakian manifold has zero torsion endomorphism with integrable vertical distribution and it is a qc-Einstein space of positive constant qc scalar curvature. Conversely, any qc-Einstein manifold with positive qc scalar curvature is locally 3-Sasakian space. 
\end{thrm}
\subsection{The Ricci identities}
We  use repeatedly the Ricci identities for the Biquard connection of order two and three,
see also \cite{IV}. Let $f$ be a smooth function on the qc
manifold $M$ with horizontal gradient $\bnab f$ defined by
$g(\bnab f,X)=df(X)$. The sub-Laplacian of $f$ is $ \triangle_b
f=-\nabla^2f(e_a,e_a)$. We have the following Ricci
identities (see e.g. \cite{IMV,IV2})
\begin{equation}\label{Ricci identities}
\begin{aligned} & \nabla^2f
(X,Y)-\nabla^2f(Y,X)=-2\sum_{s=1}^3\omega_s(X,Y)df(\xi_s), \\ &
\nabla^2f (X,\xi_s)-\nabla^2f(\xi_s,X)=T(\xi_s,X,\bnab f),\\ &
\nabla^3 f (X,Y,Z)-\nabla^3 f(Y,X,Z)=-R(X,Y,Z,\bnab f) -
2\sum_{s=1}^3 \omega_s(X,Y)\nabla^2f (\xi_s,Z),\\&
\nabla^3f(\xi_s,X,Y)-\nabla^3f(X,Y,\xi_s)=-\nabla^2f(T(\xi_s,X),Y)-\nabla^2f(X,T(\xi_s,Y))\\&\hspace{5cm}-df((\nabla_XT)(\xi_s,Y))-R(\xi_s,X,Y,\bnab f).
\end{aligned}
\end{equation}
In view of \eqref{Ricci identities} we have the decompositions (see also \cite{IPV1,IP})
\begin{equation}  \label{boh2}
\begin{aligned}  &(\bi^2f)_{[3][0]}(X,Y)=(\bi^2f)_{[3]}(X,Y) +\frac1{4n}\triangle_b fg(X,Y),\\&|(\nabla^2f)_{[3][0]}|^2=|(\nabla^2f)_{[3]}|^2-\frac{1}{4n}(\Delta_b f)^2,\\&(\bi^2f)_{[-1]}(X,Y)= (\nabla^2f)_{[-1][sym]}
(X,Y)+(\nabla^2f)_{[-1][a]}(X,Y)\\&\qquad\qquad\qquad\quad=(\nabla^2f)_{[-1][sym]}(X,Y)-\sum_{s=1}^3\omega_s(X,Y)df(\xi_s),\\ &|(\bi^2f)_{[-1]}|^2=|(\nabla^2f)_{[-1][sym]}|^2+4n|\vnab f|^2.
\end{aligned}
\end{equation}
In particular, we have
\begin{equation}  \label{xi1}
g(\nabla^2f , \omega_s) =\nabla^2f(e_a,I_se_a)=-4ndf(\xi_s).
\end{equation}
We also need  the following  qc-Bochner formula from \cite[(4.1)]{IPV1}
\begin{multline}  \label{bohS}
-\frac12\triangle_b |\bnab f|^2=|\nabla^2f|^2-g\left (\bnab
(\triangle_b f), \bnab f \right )+2(n+2)S|\bnab
f|^2+2(n+2)T^0(\bnab f,\bnab f) \\ +2(2n+2)U(\bnab f,\bnab f)+
4\sum_{s=1}^3\nabla^2f(\xi_s,I_s\bnab f), 
\end{multline}
and the  integral formulas that hold on a compact $(4n+3)$--dimensional qc manifold \cite[(3.10) and (3.12)]{IPV1}:
\begin{equation}\label{horverhess}
\begin{aligned}
&\int_M\sum_{s=1}^3\bi^2f(\xi_s,I_s\bnab f)\vol=\int_M\Big[\frac3{4n}|(\bi^2f)_{[3]}|^2-\frac1{4n}|(\bi^2f)_{[-1]}|^2-\frac12\sum_{s=1}^3\tau_{s}(I_s\bnab f,\bnab f)\Big]\vol,\\
&\int_M\sum_{s=1}^3\bi^2f(\xi_s,I_s\bnab f)\vol=-\int_M\Big[4n|\vnab f|^2
+\sum_{s=1}^3T(\xi_s,I_s\bnab f,\bnab f)\Big] \vol.
\end{aligned}
\end{equation}
\subsection{The horizontal divergence theorem}
Let $(M, g,\mathbb{Q})$ be a qc manifold of dimension $4n+3\geq 7$. For a
fixed local 1-form $\eta$ and a fixed $s\in \{1,2,3\}$ the form
$Vol_{\eta}=\eta_1\wedge\eta_2\wedge\eta_3\wedge\omega_s^{2n}$
is a locally defined volume form. Note that $Vol_{\eta}$ is independent of $%
s $ as well as the local one forms $\eta_1,\eta_2,\eta_3 $. Hence, it is a
globally defined volume form. The (horizontal) divergence of a horizontal
vector field/one-form $\sigma\in\Lambda^1\, (H)$, defined by $\nabla^*
\sigma =-tr|_{H}\nabla\sigma= -\nabla \sigma(e_a,e_a) $ 
supplies the integration by parts formula, \cite{IMV}, see also \cite{Wei},
\begin{equation*}  \label{div}
\int_M (\nabla^*\sigma)\,\, Vol_{\eta}\ =\ 0.
\end{equation*}

\section{Proof of the Theorems}
\subsection{Preliminary results}
In this subsection we prove a number of preliminary results, which serve as a base for establishing our main results. We begin with the following
\begin{lemma}\label{BochnerIn}
Let $(M,g,\mathbb{Q})$ be a qc manifold of dimension $4n+3$. Then for any $f\in\mathcal{F}(M)$ and a constant $\nu>0$ the following inequality holds:
\begin{multline}\label{BochnerIneq} 
-\frac{1}{2}\Delta_b|\bnab f|^2\geq|(\nabla^2 f)_{[-1][sym]}|^2+4n|\vnab f|^2+|(\nabla^2f)_{[3][0]}|^2+\frac{1}{4n}(\Delta_b f)^2-g(\bnab\Delta_b f,\bnab f)\\+2(n+2)S|\bnab f|^2
+2nT^0(\bnab f,\bnab f)+4(n+4)U(\bnab f,\bnab f)-\frac{6}{\nu}|\bnab f|^2-2\nu\sum_{s=1}^3[\nabla^2 f(e_a,\xi_s)]^2.
\end{multline}
\end{lemma}
\begin{proof} We have with the help of the second identity in  \eqref{Ricci identities}, \eqref{need1} and \eqref{propt} consecutively
\begin{multline}\label{HorVerRep}
\sum_{s=1}^3\nabla^2f(\xi_s,I_s\bnab f)=\sum_{s=1}^3\big[\nabla^2f(I_s\bnab f,\xi_s)-T(\xi_s,I_s\bnab f,\bnab f)\big]\\=\sum_{s=1}^3g\big(I_s\bnab f, \bnab(\xi_sf)\big)-T^0(\bnab f,\bnab f)+3U(\bnab f,\bnab f).
\end{multline}
For $\nu=Const>0$, the Cauchy--Schwarz inequality yelds
\begin{multline}\label{CSIneq}
\sum_{s=1}^3g\big(I_s\bnab f,\bnab(\xi_sf)\big)\geq-\sum_{s=1}^3|\bnab f||\bnab(\xi_sf)|\geq-\frac{1}{2}\sum_{s=1}^3\Big[\frac{1}{\nu}|\bnab f|^2+\nu|\bnab(\xi_sf)|^2\Big]\\=-\frac{3}{2\nu}|\bnab f|^2-\frac{\nu}{2}\sum_{s=1}^3|\bnab(\xi_sf)|^2.
\end{multline}
Insert \eqref{HorVerRep} and \eqref{CSIneq} into \eqref{bohS}, taking into account the decompositions $|\nabla^2f|^2=|(\nabla^2f)_{[-1]}|^2+|(\nabla^2f)_{[3]}|^2$ and \eqref{boh2}, to get \eqref{BochnerIn}.
\end{proof}

\begin{lemma}\label{Auxiliary3}
Let $(M,g,\mathbb{Q})$ be a qc manifold and $f\in\mathcal{F}(M)$. Then the following equality holds:
\begin{equation}\label{aux3}
(\nabla^3f)(e_a,e_a,\vnab f)=-\vnab f(\Delta_b f)+2g(T_{\vnab f},\nabla^2 f)-(\nabla_{e_a}T)(e_a,\vnab f,\bnab f)+Ric(\vnab f,\bnab f).
\end{equation}
\end{lemma}
\begin{proof}
We substitute $X=e_a, Y=e_a$ into the last  identity in \eqref{Ricci identities}, multiply the both sides by $df(\xi_s)$, take the sum over $s$ of the obtained identity  using the first identity in \eqref{Ricci identities} and the properties of the torsion listed in \eqref{propt} and \eqref{need1} to get  \eqref{aux3} after some standard calculations.
\end{proof}
\begin{lemma}\label{Auxiliary1}
Let $(M,g,\mathbb{Q})$ be a qc manifold  and $u(x,t)$ be a positive solution of \eqref{qcheat}. Then $f(x,t)=\ln u(x,t)$ satisfies 
\begin{equation}\label{aux1}
(\nabla^3f)(e_a,e_a,\vnab f)-\vnab f\big(\partt f\big)=-2(\nabla^2f)(\bnab f,\vnab f)+V(f),
\end{equation}
where the operator $V:\mathcal{F}(M)\rightarrow\mathcal{F}(M)$  is defined by
\begin{equation}\label{Voperator}
V(\varphi)\eqdef2T(\vnab\varphi,\bnab\varphi,\bnab\varphi)+2g(T_{\vnab\varphi},\nabla^2\varphi)-(\nabla_{e_a}T)(e_a,\vnab\varphi,\bnab\varphi)+Ric(\vnab\varphi,\bnab\varphi).
\end{equation}
\end{lemma}
\begin{proof}
It easy to see that $f(x,t)$ determined in the condition of the present lemma satisfies the equality
\begin{equation}\label{nablagradf}
\big(\Delta_b+\partt\big)f=|\bnab f|^2.
\end{equation}
Furthermore, we have the following chain of equalities:
\begin{multline*}
(\nabla^3f)(e_a,e_a,\vnab f)-\vnab f\big(\partt f\big)\\=-\vnab f\big(\partt f\big)-\vnab f(\Delta_b f)+2g(T_{\vnab f},\nabla^2 f)-(\nabla_{e_a}T)(e_a,\vnab f,\bnab f)+Ric(\vnab f,\bnab f)\\=-\vnab f(|\bnab f|^2)+2g(T_{\vnab f},\nabla^2 f)-(\nabla_{e_a}T)(e_a,\vnab f,\bnab f)+Ric(\vnab f,\bnab f)\\=-2df(\xi_s)df(e_a)[\nabla^2f(e_a,\xi_s)-T(\xi_s,e_a,\bnab f)]+2g(T_{\vnab f},\nabla^2 f)-(\nabla_{e_a}T)(e_a,\vnab f,\bnab f)+Ric(\vnab f,\bnab f)\\=-2(\nabla^2f)(\bnab f,\vnab f)+2T(\vnab f,\bnab f,\bnab f)+2g(T_{\vnab f},\nabla^2 f)-(\nabla_{e_a}T)(e_a,\vnab f,\bnab f)+Ric(\vnab f,\bnab f)\\=-2(\nabla^2f)(\bnab f,\vnab f)+V(f),
\end{multline*}
where we use \eqref{aux3} in order to get the first equality and apply  \eqref{nablagradf} to achieve the second one. The third equality is a result of the application of the second identity from \eqref{Ricci identities}.
\end{proof}

We shall use repeatedly the following 
\begin{lemma}\label{Auxiliary2}
Let $(M,g,\mathbb{Q})$ be a qc manifold, $u(x,t):M\times[0,\infty)\rightarrow\mathbb{R}$ be a positive smooth function and $f(x,t)=\ln u(x,t).$ Suppose that
\begin{equation}\label{sublapxicom}
(\nabla^3u)(e_a,e_a,\vnab u)=-\vnab u(\Delta_b u).
\end{equation}
Then $V(f)=0.$
\end{lemma}
\begin{proof}
Indeed, we have  from \eqref{Voperator}
\begin{multline*}
V(f)\\=2T(\vnab\ln u,\bnab \ln u,\bnab\ln u)+2g(T_{\vnab\ln u},\nabla^2\ln u)-(\nabla_{e_a}T)(e_a,\vnab\ln u,\bnab\ln u)+Ric(\vnab\ln u,\bnab\ln u)\\=\frac{2}{u^3}T(\vnab u,\bnab u,\bnab u)+2T(\vnab\ln u,e_a,e_b)(\nabla^2\ln u)(e_a,e_b)-\frac{1}{u^2}(\nabla_{e_a}T)(e_a,\vnab u,\bnab u)+\frac{1}{u^2}Ric(\vnab u,\bnab u)\\=\frac{2}{u^3}T(\vnab u,\bnab u,\bnab u)+\frac{2}{u}T(\vnab u,e_a,e_b)\Big[-\frac{du(e_a)du(e_b)}{u^2}+\frac{1}{u}(\nabla^2u)(e_a,e_b)\Big]\\-\frac{1}{u^2}(\nabla_{e_a}T)(e_a,\vnab u,\bnab u)+\frac{1}{u^2}Ric(\vnab u,\bnab u)\\=\frac{1}{u^2}\big[2g(T_{\vnab u},\bnab u)-(\nabla_{e_a}T)(e_a,\vnab u,\bnab u)+Ric(\vnab u,\bnab u)\big]=0,
\end{multline*}
in view of \eqref{aux3} and \eqref{sublapxicom}.
\end{proof}
Next, following the Riemannian and CR cases, we define the \emph{test function} $$F(x,t,a,c):M\times[0,\infty)\times\mathbb{R}^*\times\mathbb{R}^+\rightarrow\mathbb{R}:$$ 
\begin{equation}\label{testF}
F(x,t,a,c)=t\Big(|\bnab f|^2+a\partt f+ct|\vnab f|^2\Big),
\end{equation}
where $f(x,t)=\ln u(x,t)$ for a positive solution $u(x,t)$ of \eqref{qcheat} and the constants $a, c$ are fixed.
\begin{lemma}\label{Auxiliary4}
Let $(M,g,\mathbb{Q})$ be a compact qc manifold of dimension $4n+3$ and the a priori condition \eqref{PosA}  with $k\geq 0$ holds. Suppose that $u(x,t)$ is a positive solution of \eqref{qcheat}. Then $f(x,t)=\ln u(x,t)$ satisfies
\begin{multline}\label{aux4}
\Big(-\Delta_b-\partt\Big)F\geq-\frac{1}{t}F-2g(\bnab f,\bnab F)\\+t\Big[2|(\nabla^2f)_{[-1][sym]}|^2+(8n-c)|\vnab f|^2+\frac{1}{2n}(\Delta_b f)^2-2\Big(k+\frac{12}{ct}\Big)|\bnab f|^2+2ctV(f)\Big].
\end{multline}
\end{lemma}
\begin{proof}
It is straighforward to check afrter a differentiation of \eqref{testF} with respect to $t$ and using \eqref{nablagradf} that the following identity holds for $t>0$
\begin{equation}\label{PartF}
\partt F=\frac{1}{t}F+t\Big[2(a+1)g\Big(\bnab f,\bnab\partt f\Big)-a\Delta_b\partt f+c|\vnab f|^2+2ct\sum_{s=1}^3df(\xi_s)\partt df(\xi_s)\Big].
\end{equation}
Furthermore, we obtain from \eqref{testF}, taking into account \eqref{PosA} with $k\geq 0$ and \eqref{BochnerIneq}, the inequality
\begin{multline*}
-\Delta_b F=t\Big[-\Delta_b|\bnab f|^2-a\Delta_b\partt f-ct\Delta_b|\vnab f|^2\Big]\\\geq t\Big[2|(\nabla^2f)_{[-1][sym]}|^2+8n|\vnab f|^2+\frac{1}{2n}(\Delta_b f)^2-2g(\bnab\Delta_b f,\bnab f)-2\Big(k+\frac{6}{\nu}\Big)|\bnab f|^2-a\Delta_b\partt f\\-2(2\nu-ct)\sum_{s=1}^3[(\nabla^2f)(e_a,\xi_s)]^2+2ct\sum_{s=1}^3df(\xi_s)\nabla^3f(e_a,e_a,\xi_s)\Big].
\end{multline*}
We substitute $\nu=\frac{ct}{2}$ into the upper inequality, which, combined with \eqref{PartF}, gives
\begin{multline}\label{PartDeltaF}
\Big(-\Delta_b-\partt\Big)F\geq-\frac{1}{t}F+t\Big\{2|(\nabla^2f)_{[-1][sym]}|^2+(8n-c)|\vnab f|^2+\frac{1}{2n}(\Delta_b f)^2-2g(\bnab\Delta_b f,\bnab f)\\-2\Big(k+\frac{12}{ct}\Big)|\bnab f|^2+2ct\sum_{s=1}^3df(\xi_s)\Big[(\nabla^3f)(e_a,e_a,\xi_s)-\partt df(\xi_s)\Big]-2(a+1)g\Big(\bnab f,\bnab\partt f\Big)\Big\}.
\end{multline}
Next, we have
\begin{multline*}
-2g(\bnab\Delta_b f,\bnab f)+2ct\sum_{s=1}^3df(\xi_s)\Big[(\nabla^3f)(e_a,e_a,\xi_s)-\partt df(\xi_s)\Big]-2(a+1)g\Big(\bnab f,\bnab\partt f\Big)\\=-2g(\bnab|\bnab f|^2,\bnab f)-4ct(\nabla^2 f)(\bnab f,\vnab f)+2ctV(f)-2ag\Big(\bnab f,\bnab\partt f\Big)\\=-2g(\bnab|\bnab f|^2,\bnab f)-4ct(\nabla^2f)(\bnab f,\vnab f)+2ctV(f)-2g\bigg(\bnab f,\bnab\Big(\frac{F}{t}-|\bnab f|^2-ct|\vnab f|^2\Big)\bigg)\\=2ctV(f)-\frac{2}{t}g(\bnab f,\bnab F),
\end{multline*}
where we used \eqref{aux1} and \eqref{nablagradf} in order to obtain the first equality and the definition \eqref{testF} of $F$ to attain the second one. The third equality is obvious.

Finally, the substitution of the last equality into \eqref{PartDeltaF} proves the lemma.
\end{proof}
\begin{lemma}\label{Auxiliary5}Let $(M,g,\mathbb{Q})$ be a compact qc manifold of dimension $4n+3$ and the a priori condition \eqref{PosA} with $k\geq 0$ holds. Suppose that $u(x,t)$ is a positive solution of \eqref{qcheat}. Then $f(x,t)=\ln u(x,t)$ satisfies
\begin{multline}\label{aux5}
\Big(-\Delta_b-\partt\Big)F\geq\frac{1}{2na^2t}F(F-2na^2)-2g(\bnab f,\bnab F)\\+t\Big[2|(\nabla^2f)_{[-1][sym]}|^2+\Big(8n-c-\frac{c}{na^2}F\Big)|\vnab f|^2+\Big(-\frac{a+1}{na^2t}F-2k-\frac{24}{ct}\Big)|\bnab f|^2+2ctV(f)\Big].
\end{multline}
\end{lemma}
\begin{proof} First, we claim the truth of the following inequality
\begin{equation}\label{Deltaf2}
(\Delta_bf)^2\geq\frac{F^2}{a^2t^2}-\frac{2(a+1)}{a^2t}F|\bnab f|^2-\frac{2c}{a^2}F|\vnab f|^2.
\end{equation}
Really, we have from \eqref{nablagradf} and \eqref{testF}
\begin{equation*}
\Delta_b f=\frac{a+1}{a}|\bnab f|^2-\frac{F}{at}+\frac{ct}{a}|\vnab f|^2,
\end{equation*}
which, squared, gives 
\begin{multline*}
(\Delta_b f)^2=\frac{F^2}{(at)^2}+\Big(\frac{a+1}{a}|\bnab f|^2+\frac{ct}{a}|\vnab f|^2\Big)^2-\frac{2F}{at}\Big(\frac{a+1}{a}|\bnab f|^2+\frac{ct}{a}|\vnab f|^2\Big)\\\geq\frac{F^2}{(at)^2}-\frac{2(a+1)}{a^2t}F|\bnab f|^2-\frac{2c}{a^2}F|\vnab f|^2,
\end{multline*}
that is exactly \eqref{Deltaf2}. The substitution of \eqref{Deltaf2} into \eqref{aux4} gives \eqref{aux5}, which completes the proof of the lemma.
\end{proof}
Next, let $a, c$ and $\Theta\in[0,\infty)$ be fixed and for any $t\in[0,\Theta], \big(p(t),s(t)\big)\in M\times[0,t]$ be the point where $F$ attains its maximum on $M\times[0,t],$ i.e.
\begin{equation*}\label{maxpoint}
F(p(t),s(t),a,c)=\max_{(x,\mu)\in M\times[0,t]}F(x,\mu,a,c).
\end{equation*}
We have the following relations:
\begin{equation}\label{extremal}
\bnab F(p(t),s(t),a,c)=0,\quad-\Delta_b F(p(t),s(t),a,c)\leq 0\quad\textnormal{and}\quad\partt F(p(t),s(t),a,c)\geq 0.
\end{equation}
The first relation in \eqref{extremal} follows from the circumstance that  $\big(p(t),s(t)\big)$ is a critical point of $F(x,\mu,a,c).$ The second one is a consequence of the maximum principle for sub-elliptic operators (see e.g. \cite{BBB}), whereas the last one follows from the fact that $F(p(t),s(t),a,c)$ must be nondecreasing function in $t$. 

The substitution of \eqref{extremal} into \eqref{aux5}, considered at the point $(p(t),s(t))$, leads to the following crucial observation:
\begin{prop}\label{crucial}
Let $(M,g,\mathbb{Q})$ be a compact qc manifold of dimension $4n+3$ and the a priori condition \eqref{PosA} with $k\geq 0$ holds. Suppose that $u(x,t)$ is a positive solution of \eqref{qcheat}, $f(x,t)=\ln u(x,t)$ and $\Theta\in[0,\infty)$ is fixed. Then we have the following inequality at $(p(t),s(t))\in M\times[0,t], t\in[0,\Theta]:$
\begin{multline}\label{CrucialP}
0\geq\frac{1}{2na^2s(t)}F(F-2na^2)\\+s(t)\Big[2|(\nabla^2f)_{[-1][sym]}|^2+\Big(8n-c-\frac{c}{na^2}F\Big)|\vnab f|^2+\Big(-\frac{a+1}{na^2s(t)}F-2k-\frac{24}{cs(t)}\Big)|\bnab f|^2+2cs(t)V(f)\Big].
\end{multline}
\end{prop}
\subsection{Proof of Theorem~\ref{thrm1}} The proof relies on the observation that the following inequaity holds
\begin{equation}\label{IneqTh1}
F\big(p(\Theta),s(\Theta),-\alpha,c\big)<\frac{4(9+2n)^2}{3c}=\frac{16n^2\alpha^2}{3c},\quad\textnormal{provided}\quad 0<c<\frac{8n}{3}.
\end{equation}
We shall prove \eqref{IneqTh1} by contradiction. Indeed, suppose that 
\begin{equation}\label{NotIneqTh1}
F\big(p(\Theta),s(\Theta),-\alpha,c\big)\geq\frac{16n^2\alpha^2}{3c}. 
\end{equation}
Since $F\big(p(t),s(t),-\alpha,c\big)$ is continuous in $t\in[0,\Theta]$ and $F\big(p(0),s(0),-\alpha,c\big)=0,$ we get from \eqref{NotIneqTh1} the existence of a $t_0\in(0,\Theta],$ s.t.
\begin{equation*}\label{NotIneqTh1Ex}
F\big(p(t_0),s(t_0),-\alpha,c\big)=\frac{16n^2\alpha^2}{3c}.
\end{equation*}
Now we take \eqref{CrucialP} at the point $(p(t_0),s(t_0))$, under the assumption that $a=-\alpha, k=0$ and \eqref{commut} hold (the latter yelds $V(f)=0$, due to Lemma~\ref{Auxiliary2}), which, together with the upper equality, leads after some standard calculations to the inequality
\begin{equation*}
0\geq\frac{16n^2\alpha^2(8n-3c)}{9c^2s(t_0)}+s(t_0)\frac{8n-3c}{3}|\vnab f|^2.
\end{equation*}
But the right--hand side  of this inequality is strictly greater than $0,$ provided $0<c<\frac{8n}{3},$ which leads to a contradiction. Hence, the assumption \eqref{NotIneqTh1} fails and  \eqref{IneqTh1} holds. In other words, we have 
\begin{equation*}
\max_{(x,\mu)\in M\times[0,\Theta]}\mu\Big(|\bnab f|^2-\alpha\partt f+c\mu|\vnab f|^2\Big)<\frac{16n^2\alpha^2}{3c},
\end{equation*}
or, if we restrict our considerations on the set $M\times\{\Theta\},$ we get 
\begin{equation*}
\Theta\Big(|\bnab f|^2-\alpha\partt f+c\Theta|\vnab f|^2\Big)<\frac{16n^2\alpha^2}{3c}.
\end{equation*}
Finally, since the upper inequality holds for any $\Theta\geq 0$, we get after setting $\Theta=t>0$ and letting $c\rightarrow\frac{8n}{3}$ the needed estimate \eqref{sub1}, which ends the proof of Theorem~\ref{thrm1}.
\subsection{Proof of Corollary~\ref{cor1}} First we show the following
\begin{lemma}\label{lc2}
On a qc-Einstein space $(M,g,\mathbb{Q})$ the condition  \eqref{commut} holds true.
\end{lemma}
\begin{proof}
Let  $(M,g,\mathbb{Q})$ be a qc-Einstein space. Theorem~\ref{3sas} tells us that that the torsion endomorphosm $T(\xi,X,Y)=0$ and the vertical distrubution is integrable. Then the fourth equality of \eqref{sixtyfour} implies $\rho_t(\xi_s,X)=0, s\not=t,$ and the fifth equality in \eqref{sixtyfour} yields $Ric(\xi_t,X)=0$. Hence, the last three terms of the right--hand side of \eqref{aux3} vanish and we get \eqref{commut} holds which proofs the lemma.
\end{proof}
Further, Lemma~\ref{lc2} shows that  on a qc-Einstein manifold the conditions of Theorem \ref{thrm1} are satisfied and the estimate \eqref{sub1} follows. The proof of  Corollary~\ref{cor1} is completed.
\subsection{Proof of Theorem~\ref{thrm2} and Corollary~\ref{cor3}} Following the CR case, we shall separate our proof into two cases.

\textbf{Case I: $\Theta>\frac{3+k}{8n}.$} We claim in this case that 
\begin{equation}\label{IneqTh3}
F\Big(p(\Theta),s(\Theta),-\alpha_k,\frac{1}{\Theta}\Big)<\frac{8n^2\alpha_k^2(k+2)}{k+3}\Theta. 
\end{equation}
Just as in the proof of Theorem~\ref{thrm1}, we shall establish \eqref{IneqTh3} by contradiction. Namely, suppose that 
\begin{equation}\label{NotIneqTh3}
F\Big(p(\Theta),s(\Theta),-\alpha_k,\frac{1}{\Theta}\Big)\geq\frac{8n^2\alpha_k^2(k+2)}{k+3}\Theta.
\end{equation}
Since $F\Big(p(t),s(t),-\alpha_k,\frac{1}{\Theta}\Big)$ is continuous in $t\in[0,\Theta]$ and $F\Big(p(0),s(0),-\alpha_k,\frac{1}{\Theta}\Big)=0,$ we obtain from \eqref{NotIneqTh3} the existence of a $t_0\in(0,\Theta],$ s.t.
\begin{equation}\label{NotIneqTh3Ex}
F\Big(p(t_0),s(t_0),-\alpha_k,\frac{1}{\Theta}\Big)=\frac{8n^2\alpha_k^2(k+2)}{k+3}\Theta.
\end{equation}
Now we consider \eqref{CrucialP} at the point $(p(t_0),s(t_0)),$ assuming that $a=-\alpha_k, c=\frac{1}{\Theta}$ and \eqref{commut} holds. Hence,  $V(f)=0$ according to Lemma~\ref{Auxiliary2}. Then  \eqref{NotIneqTh3Ex} substituted into  \eqref{CrucialP} implies, after some straightforward computations,  the inequality 
\begin{multline}\label{ContraTh3} 
0\geq\frac{8n^2\alpha_k^2(k+2)\Theta}{s(t_0)(k+3)}\Big(\frac{4n(k+2)}{k+3}\Theta-1\Big)\\+s(t_0)\Big(8n-\frac{1}{\Theta}-\frac{8n(k+2)}{k+3}\Big)|\vnab f|^2+s(t_0)\Big(\frac{4(9+2nk)(k+2)}{s(t_0)(k+3)}\Theta-2k-\frac{24\Theta}{s(t_0)}\Big)|\bnab f|^2.
\end{multline}
It is easy to check that $$\frac{4n(k+2)}{k+3}\Theta-1>0,\quad 8n-\frac{1}{\Theta}-\frac{8n(k+2)}{k+3}>0\quad\textnormal{and}\quad\frac{4(9+2nk)(k+2)}{s(t_0)(k+3)}\Theta-2k-\frac{24\Theta}{s(t_0)}>0,$$ for $\Theta>\frac{3+k}{8n}.$ The latter combined with \eqref{ContraTh3} leads to  a contradiction, and hence \eqref{NotIneqTh3} fails, i.e. \eqref{IneqTh3} holds. Thus we have $$\max_{(x,\mu)\in M\times[0,\Theta]}\mu\Big(|\bnab f|^2-\alpha_k\partt f+\frac{\mu}{\Theta}|\vnab f|^2\Big)<\frac{8n^2\alpha_k^2(k+2)}{k+3}\Theta,$$
and if we shrink on the set $M\times\{\Theta\},$ we get $$\Theta\Big(|\bnab f|^2-\alpha_k\partt f +|\vnab f|^2\Big)<\frac{8n^2\alpha_k^2(k+2)}{k+3}\Theta,$$
and hence
\begin{equation}\label{caseI}
|\bnab f|^2-\alpha_k\partt f<\frac{8n^2\alpha_k^2(k+2)}{k+3},\quad\textnormal{provided}\quad t>\frac{3+k}{8n}.
\end{equation}

\textbf{Case II: $\Theta\leq\frac{3+k}{8n}.$} As before, we assert that 
\begin{equation}\label{IneqTh3A}
F\big(p(\Theta),s(\Theta),-\alpha_k,c\big)<\frac{8n^2\alpha_k^2(k+2)}{(k+3)c},\quad\textnormal{for}\quad 0<c<\frac{8n}{k+3},
\end{equation}
and we will prove it again by contradiction. Suppose that 
\begin{equation}\label{NotIneqTh3A}
F\big(p(\Theta),s(\Theta),-\alpha_k,c\big)\geq\frac{8n^2\alpha_k^2(k+2)}{(k+3)c},\quad\textnormal{for}\quad 0<c<\frac{8n}{k+3}.
\end{equation}
Because of $F\big(p(t),s(t),-\alpha_k,c\big)$ is continuous in $t\in[0,\Theta]$ and $F\big(p(0),s(0),-\alpha_k,c\big)=0,$ we have from \eqref{NotIneqTh3A} that there exists a $t_0\in(0,\Theta]$, s.t.
\begin{equation}\label{IneqTh3AEx}
F\big(p(t_0),s(t_0),-\alpha_k,c\big)=\frac{8n^2\alpha_k^2(k+2)}{(k+3)c}.
\end{equation}
We take the inequality \eqref{CrucialP} at the point $(p(t_0),s(t_0)),$ assuming $a=-\alpha_k, V(f)=0$ and  \eqref{IneqTh3AEx}, leading to the following 
\begin{multline}\label{ContraTh3A}
0\geq\frac{8n^2\alpha_k^2(k+2)}{s(t_0)(k+3)c}\Big(\frac{4n(k+2)}{(k+3)c}-1\Big)\\+s(t_0)\Big(8n-c-\frac{8n(k+2)}{k+3}\Big)|\vnab f|^2+s(t_0)\Big(\frac{4(9+2nk)(k+2)}{s(t_0)(k+3)c}-2k-\frac{24}{cs(t_0)}\Big)|\bnab f|^2.
\end{multline}
Since  $0<c<\frac{8n}{k+3}$ and hence $s(t_0)c<1$, we  obtain that
$$\frac{4n(k+2)}{(k+3)c}-1>0,\quad 8n-c-\frac{8n(k+2)}{k+3}>0\quad\textnormal{and}\quad\frac{4(9+2nk)(k+2)}{s(t_0)(k+3)c}-2k-\frac{24}{cs(t_0)}>0.$$
These inequalities together with \eqref{ContraTh3A} bring us to a  contradiction with \eqref{NotIneqTh3A}, i.e. \eqref{IneqTh3A} holds. In other words, we have $$\max_{(x,\mu)\in M\times[0,\Theta]}\mu\Big(|\bnab f|^2-\alpha_k\partt f+\mu c|\vnab f|^2\Big)<\frac{8n^2\alpha_k^2(k+2)}{(k+3)c},$$ and applying the same argument as in Case I, we get the inequality
 $$|\bnab f|^2-\alpha_k\partt f<\frac{8n^2\alpha_k^2(k+2)}{(k+3)ct},\quad\textnormal{provided}\quad 0<t\leq\frac{3+k}{8n},\quad 0<c<\frac{8n}{3+k}.$$ 
Finally, we let $c\rightarrow\frac{8n}{3+k}$ into the above inequalilty in order to obtain
\begin{equation}\label{caseII}
|\bnab f|^2-\alpha_k\partt f\leq\frac{n\alpha_k^2(k+2)}{t},\quad\textnormal{provided}\quad 0<t\leq\frac{3+k}{8n}.
\end{equation}
The inequality \eqref{sub2} follows from 
the inequalities \eqref{caseI} and \eqref{caseII}, which ends the proof of Theorem~\ref{thrm2}.  

Corollary~\ref{cor3} is a straightforward consequence of Lemma~\ref{lc2} and Theorem~\ref{thrm2}.

\subsection{Proof of Theorem~\ref{thrm3} and Corollary~\ref{cor2}}
We begin with the following equivalent form of \eqref{heatkern}:
\begin{equation}\label{heatkernA}
f(x,t)=-\varphi(x,t)-2n\alpha^2a\ln(4\pi t),
\end{equation}
which leads immediately to the equalities
\begin{equation}\label{partfnablaf}
\partt f(x,t)=-\partt\varphi(x,t)-\frac{2n\alpha^2a}{t},\quad|\bnab f|^2=|\bnab\varphi|^2\quad\textnormal{and}\quad|\bnab u|^2=u^2|\bnab\varphi|^2.
\end{equation}
We get from Corollary~\ref{cor1}
\begin{equation*}\label{IneqCons}
|\bnab f|^2-\alpha\partt f\leq\frac{2n\alpha^2}{t},
\end{equation*}
which can be written, taking into account the first two equalities in \eqref{partfnablaf}, into the form
\begin{equation}\label{Ineq}
|\bnab\varphi|^2+\alpha\partt\varphi+\frac{2n\alpha^2(\alpha a-1)}{t}\leq 0.
\end{equation}
Now we claim that
\begin{equation}\label{dertN}
\frac{d}{dt}\Ncal(u,t)=\int_Mu|\bnab\varphi|^2\vol.
\end{equation}
Indeed, we have the chain of identities
\begin{multline*}
\frac{d}{dt}\Ncal(u,t)=-\int_M\Big(\partt u+(\ln u)\partt u\Big)\vol=\int_M\Delta_b u\ln u\vol\\=\int_Mu\Delta_b\ln u\vol=\int_M\frac{|\bnab u|^2}{u}\vol=\int_M u|\bnab\varphi|^2\vol,
\end{multline*}
where we used 
\eqref{qcheat}, the self-adjointness of the sub-Laplacian and the third identity in \eqref{partfnablaf}.

Furthermore, we claim the validity of  the next identity 
\begin{equation}\label{dertTN}
\frac{d}{dt}\tilde{\Ncal}(u,t)=\int_M\Big(|\bnab\varphi|^2+\alpha\partt\varphi+\frac{2n\alpha^2(\alpha-1)a}{t}\Big)u\vol.
\end{equation}
 Indeed, taking into account 
\eqref{qcheat}, \eqref{normu}, the first equality in \eqref{partfnablaf} and \eqref{dertN}, we calculate 
\begin{multline*}
\frac{d}{dt}\tilde{\Ncal}(u,t)=\frac{d}{dt}\Ncal(u,t)-\frac{2n\alpha^2a}{t}=\int_Mu|\bnab\varphi|^2\vol-\frac{2n\alpha^2a}{t}=\int_M\Big(|\bnab\varphi|^2u-\frac{2n\alpha^2a}{t}u\Big)\vol\\=\int_M\bigg(|\bnab\varphi|^2u+\Big(\partt f+\partt\varphi\Big)u\bigg)\vol=\int_M\Big(|\bnab\varphi|^2u+\partt u+\partt\varphi u\Big)\vol=\int_M\Big(|\bnab\varphi|^2+\partt\varphi\Big)u\vol\\=\int_M\Big(|\bnab\varphi|^2+\alpha\partt\varphi+(1-\alpha)\partt\varphi\Big)u\vol=\int_M\bigg(|\bnab\varphi|^2+\alpha\partt\varphi+(1-\alpha)\Big(\partt f-\frac{2n\alpha^2a}{t}\Big)\bigg)u\vol\\=\int_M\Big(|\bnab\varphi|^2+\alpha\partt\varphi+\frac{2n\alpha^2(\alpha-1)a}{t}\Big)u\vol.
\end{multline*}
Suppose  $a\geq 1$ and hence $$\frac{2n\alpha^2(\alpha-1)a}{t}\leq\frac{2n\alpha^2(\alpha a-1)}{t}.$$
The last inequality together with \eqref{Ineq} leads to 
\begin{equation}\label{IneqderN}|\bnab\varphi|^2+\alpha\partt\varphi+\frac{2n\alpha^2(\alpha-1)a}{t}\leq|\bnab\varphi|^2+\alpha\partt\varphi+\frac{2n\alpha^2(\alpha a-1)}{t}\leq 0.
\end{equation}
Now, \eqref{IneqderN}  implies the entropy formula \eqref{Nashentropy} since $u$ is a positive function.  Theorem~\ref{thrm3} is proved.

\subsubsection{Proof of Corollary~\ref{cor2}} We may assume without loss of generality that $u(x,t)$ satisfies the normalization \eqref{normu}. Then we get from \eqref{Nashentropy}, \eqref{Nash1}, \eqref{Nash2}, \eqref{qcheat} and an integration by parts
\begin{multline*}
\frac{d}{dt}\tilde{\Ncal}(u,t)=\frac{d}{dt}\bigg(\Ncal(u,t)-2n\alpha^2a\big(\ln(4\pi t)+1\big)\bigg)=-\frac{d}{dt}\int_M(\ln u)u\vol-\frac{2n\alpha^2a}{t}\\=\int_M(\ln u)\Delta_b u\vol-\frac{2n\alpha^2a}{t}=\int_M\frac{|\bnab u|^2}{u}\vol-\frac{2n\alpha^2a}{t}=4\int_M|\bnab u^\frac12|^2\vol-\frac{2n\alpha^2a}{t}\leq 0,
\end{multline*}
and the proof of Corollary~\ref{cor2} is completed.

\subsection{Proof of Theorem~\ref{thrm4}}
We start with the  identity
\begin{equation}\label{WNT}
\Wcal(u,t)=\frac{d}{dt}\big(t\tilde{\Ncal}(u,t)\big).
\end{equation}
Indeed, we have by the very definitions \eqref{Nash1}, \eqref{Nash2} and \eqref{Perelman1} of $\Ncal$, $\tilde{\Ncal}$ and $\Wcal$, and also by \eqref{normu}, \eqref{heatkernA} and \eqref{dertN}, the chain of identities 
\begin{multline*}
\frac{d}{dt}\big(t\tilde{\Ncal}(u,t)\big)=\frac{d}{dt}\bigg(t\Ncal(u,t)-2n\alpha^2at\big(\ln(4\pi t)+1\big)\bigg)=\Ncal(u,t)+t\int_Mu|\bnab\varphi|^2\vol-2n\alpha^2a\ln(4\pi t)-4n\alpha^2a\\=-\int_M(\ln u)u\vol+t\int_Mu|\bnab\varphi|^2\vol-2n\alpha^2a\ln(4\pi t)-4n\alpha^2a\\=\int_M\big(\varphi+2n\alpha^2a\ln(4\pi t)\big)u\vol+t\int_Mu|\bnab\varphi|^2\vol-2n\alpha^2a\ln(4\pi t)-4n\alpha^2a\\=\int_M\big(\varphi+t|\bnab\varphi|^2-4n\alpha^2a\big)u\vol=\Wcal(u,t).
\end{multline*}
Furthermore, we get immediately from \eqref{WNT} the equality
\begin{equation}\label{dertW}
\frac{d}{dt}\Wcal(u,t)=2\frac{d}{dt}\tilde{\Ncal}(u,t)+t\frac{d^2}{dt^2}\tilde{\Ncal}(u,t).
\end{equation}
Our next aim is to find a suitable estimation of the derivative $\frac{d}{dt}\Wcal(u,t),$ using the representation \eqref{dertW}. More precisely, we claim  the following inequality holds
\begin{multline}\label{dertWIneq}
\frac{d}{dt}\Wcal(u,t)\leq2\int_Mu|\bnab\varphi|^2\vol-2t\int_M|\nabla^2\ln u|^2u\vol-4t\Big((n+2)S-\frac{3}{\nu}\Big)\int_M|\bnab\ln u|^2u\vol\\+4\nu t\int_M\sum_{s=1}^3\big((\nabla^2\ln u)(e_a,\xi_s)\big)^2u\vol-\frac{2n\alpha^2a}{t},
\end{multline}
where $\nu>0$ is a constant.
In order to prove \eqref{dertWIneq}, we take into account the  identities
\begin{multline}\label{dsecondtN}
\frac{d^2}{dt^2}\tilde{\Ncal}(u,t)=\frac{d}{dt}\int_M\Big(|\bnab\varphi|^2-\frac{2n\alpha^2a}{t}\Big)u\vol=\frac{d}{dt}\int_M\Big(\frac{|\bnab u|^2}{u^2}-\frac{2n\alpha^2a}{t}\Big)u\vol\\=\frac{d}{dt}\int_M\Big(\Delta_b\ln u-\frac{2n\alpha^2a}{t}\Big)u\vol=\int_M\Big(\partt(\Delta_b\ln u)u+\Delta_b\ln u\partt u\Big)\vol+\frac{2n\alpha^2a}{t^2}\\=\int_M\Big(\partt(\Delta_b\ln u)u-\Delta_b\ln u\Delta_b u\Big)\vol+\frac{2n\alpha^2a}{t^2},
\end{multline}
where we applied the very definition \eqref{Nash2} of $\tilde{\Ncal}(u,t)$, \eqref{dertN}, \eqref{partfnablaf} and \eqref{qcheat} in an obvious way.

 Moreover, we obtain by some straightforward computations, using \eqref{qcheat}, that 
\begin{equation*}
\Big(\partt+\Delta_b\Big)\Delta_b\ln u=\Delta_b|\nabla_b\ln u|^2,
\end{equation*}
which, substituted into \eqref{dsecondtN} and taking into account the self-adjointness of $\Delta_b$, leads to
\begin{equation}\label{dersecondtN}
\frac{d^2}{dt^2}\tilde{\Ncal}(u,t)=\int_M\Big(\Delta_b|\nabla_b\ln u|^2u-2\Delta_b\ln u\Delta_b u\Big)\vol+\frac{2n\alpha^2a}{t^2}.
\end{equation} 
For a positive constant $\nu$ we claim that
\begin{multline}\label{BochnerLn}
-\frac{1}{2}\Delta_b|\nabla_b\ln u|^2\\\geq|\nabla^2\ln u|^2-g(\nabla_b\Delta_b\ln u,\nabla_b\ln u)+2\Big((n+2)S-\frac{3}{\nu}\Big)|\nabla_b\ln u|^2-2\nu\sum_{s=1}^3\big((\nabla^2\ln u)(e_a,\xi_s)\big)^2.
\end{multline}
 Indeed, we get from \eqref{bohS}, having in mind $T^0=U=0$, the following 
\begin{multline*}
-\frac{1}{2}\Delta_b|\nabla_b\ln u|^2=|\nabla^2\ln u|^2-g(\nabla_b\Delta_b\ln u,\nabla_b\ln u)+2(n+2)S|\nabla_b\ln u|^2+4\sum_{s=1}^3(\nabla^2\ln u)(\xi_s, I_s\nabla_b\ln u)\\=|\nabla^2\ln u|^2-g(\nabla_b\Delta_b\ln u,\nabla_b\ln u)+2(n+2)S|\nabla_b\ln u|^2+4\sum_{s=1}^3(\nabla^2\ln u)(I_s\nabla_b\ln u,\xi_s)\\=|\nabla^2\ln u|^2-g(\nabla_b\Delta_b\ln u,\nabla_b\ln u)+2(n+2)S|\nabla_b\ln u|^2+4\sum_{s=1}^3g\big(\nabla_b(\xi_s\ln u),I_s\nabla_b\ln u\big)\\\geq|\nabla^2\ln u|^2-g(\nabla_b\Delta_b\ln u,\nabla_b\ln u)+2(n+2)S|\nabla_b\ln u|^2-4\sum_{s=1}^3||\nabla_b(\xi_s\ln u)||||I_s\nabla_b\ln u||\\\geq|\nabla^2\ln u|^2-g(\nabla_b\Delta_b\ln u,\nabla_b\ln u)+2(n+2)S|\nabla_b\ln u|^2-2\nu\sum_{s=1}^3g\big(\nabla_b(\xi_s\ln u),\nabla_b(\xi_s\ln u)\big)\\-\frac{6}{\nu}g(\nabla_b\ln u,\nabla_b\ln u)=|\nabla^2\ln u|^2-g(\nabla_b\Delta_b\ln u,\nabla_b\ln u)+2\Big((n+2)S-\frac{3}{\nu}\Big)|\nabla_b\ln u|^2\\-2\nu\sum_{s=1}^3\big((\nabla^2\ln u)(e_a,\xi_s)\big)^2,
\end{multline*}
which is exactly \eqref{BochnerLn}. Note that we used the second identity in \eqref{Ricci identities} in order to get the second equality in the above chain. The first inequality is obtained by an application of the Cauchy--Schwarz inequality, while the second inequality is obvious. Now we substitute \eqref{BochnerLn} into \eqref{dersecondtN} to get
\begin{multline}\label{dersecondtNLast}
\frac{d^2}{dt^2}\tilde{\Ncal}(u,t)\leq\int_M\Big\{\Big[-2|\nabla^2\ln u|^2+2g(\nabla_b\Delta_b\ln u,\nabla_b\ln u)-4\Big((n+2)S-\frac{3}{\nu}\Big)|\nabla_b\ln u|^2\\+4\nu\sum_{s=1}^3\big((\nabla^2\ln u)(e_a,\xi_s)\big)^2\Big]u-2\Delta_b\ln u\Delta_bu\Big\}\vol+\frac{2n\alpha^2a}{t^2}\\=\int_M\Big[-2|\nabla^2\ln u|^2-4\Big((n+2)S-\frac{3}{\nu}\Big)|\nabla_b\ln u|^2+4\nu\sum_{s=1}^3\big((\nabla^2\ln u)(e_a,\xi_s)\big)^2\Big]u\vol+\frac{2n\alpha^2a}{t^2},
\end{multline}
where we used the identity 
\begin{equation}\label{intp}\int_M\Delta_b\ln u\Delta_bu\vol=\int_Mg(\nabla_b\Delta_b\ln u,\nabla_b\ln u)u\vol
\end{equation} in order to conclude  the equality above. Of course, \eqref{intp} is obvious, due to an integration by parts.
Finally, we substitute \eqref{dersecondtNLast} into \eqref{dertW} and take into account the definition \eqref{Nash2} of $\tilde{\Ncal}(u,t)$ and \eqref{dertN} to get \eqref{dertWIneq}.

Furthermore, we shall need the following inequality
\begin{equation}\label{HessLn}
|\nabla^2\ln u|^2\geq|(\nabla^2\varphi)_{[-1][sym]}|^2+4n|\nabla_v\varphi|^2+\frac{1}{4n}(\Delta_b\varphi)^2,
\end{equation}
which is a result of \eqref{boh2} and \eqref{heatkernA}.

Next, we have from the definition \eqref{Perelman2} of $\tilde{\Wcal}(u,t),$ \eqref{dertWIneq}, \eqref{qcheat}, \eqref{HessLn} and \eqref{heatkernA} consecutively
\begin{multline}\label{DertTW}
\frac{d}{dt}\tilde{\Wcal}(u,t)=\frac{d}{dt}\Wcal(u,t)+8nt\int_M|\nabla_v\varphi|^2u\vol+8nt^2\int_Md\varphi_t(\nabla_v\varphi)u\vol+4nt^2\int_M|\nabla_v\varphi|^2u_t\vol\\\leq2\int_Mu|\nabla_b\varphi|^2\vol-2t\int_M|\nabla^2\ln u|^2u\vol-4t\Big((n+2)S-\frac{3}{\nu}\Big)\int_M|\nabla_b\ln u|^2u\vol\\+4\nu t\int_M\sum_{s=1}^3\big((\nabla^2\ln u)(e_a,\xi_s)\big)^2u\vol-\frac{2n\alpha^2a}{t}+8nt\int_M|\nabla_v\varphi|^2u\vol+8nt^2\int_Md\varphi_t(\nabla_v\varphi)u\vol\\-4nt^2\int_M|\nabla_v\varphi|^2\Delta_bu\vol\leq2\int_Mu|\nabla_b\varphi|^2\vol-2t\int_M\Big[|(\nabla^2\varphi)_{[-1][sym]}|^2+4n|\nabla_v\varphi|^2+\frac{1}{4n}(\Delta_b\varphi)^2\Big]u\vol\\-4t\Big((n+2)S-\frac{3}{\nu}\Big)\int_M|\nabla_b\varphi|^2u\vol+4\nu t\int_M\sum_{s=1}^3\big((\nabla^2\varphi)(e_a,\xi_s)\big)^2u\vol-\frac{2n\alpha^2a}{t}+8nt\int_M|\nabla_v\varphi|^2u\vol\\+8nt^2\int_Md\varphi_t(\nabla_v\varphi)u\vol-4nt^2\int_M|\nabla_v\varphi|^2\Delta_bu\vol=-2t\int_M|(\nabla^2\varphi)_{[-1][sym]}|^2u\vol\\-\frac{t}{2n}\int_M(\Delta_b\varphi)^2u\vol-4t\Big((n+2)S-\frac{3}{\nu}-\frac{1}{2t}\Big)\int_M|\nabla_b\varphi|^2u\vol+4\nu t\int_M\sum_{s=1}^3\big((\nabla^2\varphi)(e_a,\xi_s)\big)^2u\vol-\frac{2n\alpha^2a}{t}\\+8nt^2\int_Md\varphi_t(\nabla_v\varphi)u\vol-4nt^2\int_M|\nabla_v\varphi|^2\Delta_bu\vol.
\end{multline}
The next aim is to find a suitable representation of the term $$8nt^2\int_Md\varphi_t(\nabla_v\varphi)u\vol-4nt^2\int_M|\nabla_v\varphi|^2\Delta_bu\vol$$  appearing at the right--hand side of \eqref{DertTW}. 

First of all, we have from \eqref{heatkern}, after standard computations, the following formulas:
\begin{equation*}
\partt u=-u\Big(\partt\varphi+\frac{2n\alpha^2a}{t}\Big)\quad\textnormal{and}\quad\Delta_b u=-u(\Delta_b\varphi+|\nabla_b\varphi|^2).
\end{equation*}
Now, the substitution of the above two representations into \eqref{qcheat} gives
\begin{equation}\label{qcheatphi}
\partt\varphi=-\Delta_b\varphi-|\nabla_b\varphi|^2-\frac{2n\alpha^2a}{t}.
\end{equation}
Since $(M,g,\mathbb{Q})$ is a qc-Einstein manifold, we have  from Lemma~\ref{lc2}
\begin{equation}\label{nablathreePhi}
(\nabla^3\varphi)(e_a,e_a,\nabla_v\varphi)=-\nabla_v\varphi(\Delta_b\varphi).
\end{equation}
Furthermore, we obtain after simple calculations from \eqref{qcheatphi} and \eqref{nablathreePhi} the following equality
\begin{equation}\label{VertgradPhi}
(\nabla_v\varphi)\Big(\partt\varphi\Big)=(\nabla^3\varphi)(e_a,e_a\nabla_v\varphi)-2(\nabla^2\varphi)(\nabla_v\varphi,\nabla_b\varphi).
\end{equation}
Now, we assert the following identity 
\begin{equation}\label{LastDiff}
8nt^2\int_Md\varphi_t(\nabla_v\varphi)u\vol-4nt^2\int_M|\nabla_v\varphi|^2\Delta_bu\vol=-8nt^2\int_M\big((\nabla^2\varphi)(e_a,\xi_s)\big)^2u\vol.
\end{equation}
Really, we have successively
\begin{multline*}
8nt^2\int_Md\varphi_t(\nabla_v\varphi)u\vol-4nt^2\int_M|\nabla_v\varphi|^2\Delta_bu\vol=16nt^2\int_M(\nabla^3\varphi)(e_a,e_a,\nabla_v\varphi)u\vol\\-16nt^2\int_M(\nabla^2\varphi)(\nabla_v\varphi,\nabla_b\varphi)u\vol+8nt^2\int_M\big((\nabla^2\varphi)(e_a,\xi_s)\big)^2u\vol\\=16nt^2\int_M(\nabla^2\varphi)(e_a,\xi_s)(\nabla^2u)(e_a,\xi_s)\vol-16nt^2\int_M(\nabla^2\varphi)(\nabla_v\varphi,\nabla_b\varphi)u\vol\\+8nt^2\int_M\big((\nabla^2\varphi)(e_a,\xi_s)\big)^2u\vol=-16nt^2\int_M(\nabla^2\varphi)(\nabla_bu,\nabla_v\varphi)\vol-16nt^2\int_M\big((\nabla^2\varphi)(e_a,\xi_s)\big)^2u\vol\\-16nt^2\int_M(\nabla^2\varphi)(\nabla_v\varphi,\nabla_b\varphi)u\vol+8nt^2\int_M\big((\nabla^2\varphi)(e_a,\xi_s)\big)^2u\vol=-8nt^2\int_M\big((\nabla^2\varphi)(e_a,\xi_s)\big)^2u\vol,
\end{multline*}
which is exactly \eqref{LastDiff}. Notice that in the upper chain we took into account \eqref{VertgradPhi} and the self--adjointness of $\Delta_b$ in order to take the first identity, while the second one is a consequence of \eqref{heatkernA} and an integration by parts. The third equality follows from \eqref{heatkern}, while we got the last one by the second Ricci identity from \eqref{Ricci identities}. 

Now we substitute \eqref{LastDiff} into \eqref{DertTW}, setting $\nu=2nt,$ which leads to 
\begin{multline}\label{DertTWLast}
\frac{d}{dt}\tilde{\Wcal}(u,t)\leq-2t\int_M|(\nabla^2\varphi)_{[-1][sym]}|^2u\vol-\frac{t}{2n}\int_M(\Delta_b\varphi)^2u\vol-4t(n+2)S\int_Mu|\nabla_b\varphi|^2\vol\\+2\Big(\frac{3}{n}+1\Big)\int_Mu|\nabla_b\varphi|^2\vol-\frac{2n\alpha^2a}{t}.
\end{multline}
We have from \eqref{sub1} that 
\begin{equation}\label{int2}\frac{|\nabla_bu|^2}{u}\leq\alpha\partt u+\frac{2n\alpha^2}{t}u,
\end{equation}
and since $\nabla_bu=-u\nabla_b\varphi,$ we get  after an integration of  \eqref{int2}  that
\begin{equation*}
\int_M|\nabla_b\varphi|^2u\vol\leq\frac{2n\alpha^2}{t}.
\end{equation*}
A substitution of the last inequality into \eqref{DertTWLast} gives \eqref{Perelmanentropy}, provided $a\geq\frac{6+2n}{n},$ which completes the proof of Theorem~\ref{thrm4}.

\end{document}